\documentclass[a4paper,11pt]{amsart}

\usepackage{amsthm,amsmath,amssymb,amsfonts}
\usepackage[utf8]{inputenc}
\usepackage{verbatim}
\usepackage{marvosym}
\usepackage{stackrel}
\usepackage{graphicx}
\usepackage[svgnames]{xcolor}
\definecolor{blue(munsell)}{rgb}{0.0, 0.5, 0.69}
\usepackage{lipsum}
\usepackage{svg}
\usepackage{caption}
\usepackage{enumitem}
\usepackage{epigraph}
\usepackage{ebproof}
\usepackage[cal=boondoxo]{mathalfa}
\usepackage[all,cmtip]{xy}
\usepackage{mathtools}
\usepackage{color}
\usepackage{mathrsfs}
\usepackage{subfig}
\usepackage{framed}
\usepackage{hyperref}
\hypersetup{hypertexnames = false, bookmarksdepth = 2, bookmarksopen = true, colorlinks, linkcolor = black, citecolor = blue(munsell), urlcolor = blue(munsell), pdfstartview={XYZ null null 1}}
\usepackage{cleveref}
\usepackage{tikz}
\usepackage{pxfonts}
\usepackage{bbold}

\usepackage[colorinlistoftodos]{todonotes}
\usepackage{xparse}
\DeclareDocumentCommand\issue{g}{\todo[size=\footnotesize,color = green!40]{Issue\IfNoValueF{#1}{: #1}}}
\DeclareDocumentCommand\tobedone{g}{\todo[size=\footnotesize,color = yellow!50]{To do\IfNoValueF{#1}{: #1}}}
\DeclareDocumentCommand\notationissue{g}{\todo[size=\footnotesize,color = red!30]{Notation?\IfNoValueF{#1}{: #1}}}
\DeclareDocumentCommand\doubt{g}{\todo[size=\footnotesize,color = blue!10]{Doubt\IfNoValueF{#1}{: #1}}}
\DeclareDocumentCommand\observation{g}{\todo[size=\footnotesize,color = orange!10]{Observation\IfNoValueF{#1}{: #1}}}

\usepackage[all,cmtip]{xy}

\usetikzlibrary{calc}
\usetikzlibrary{matrix,arrows,arrows.meta}
\usepackage{tikz-cd}

\usepackage{calligra} 

\newtagform{fn}{(}{)\footnotemark}

    {\endMakeFramed}

    {\endMakeFramed}

\makeatletter
\g@addto@macro\bfseries{\boldmath}
\makeatother

\theoremstyle{definition}
\newtheorem{thm}{Theorem}[section]
\newtheorem*{thm*}{Theorem}
\newtheorem{prop}[thm]{Proposition}

\newtheorem{cor}[thm]{Corollary}
\newtheorem{defn}[thm]{Definition}
\newtheorem{rem}[thm]{Remark}
\newtheorem{ass}[thm]{Assumption}
\newtheorem{exa}[thm]{Example}
\newtheorem{conj}[thm]{Conjecture}
\newtheorem{notat}[thm]{Notation}

\newcommand\lan{\mathsf{lan}}
\newcommand\ran{\mathsf{ran}}

\newcommand\Set{\operatorname{\bf Set}}

\newcommand\ca{\mathsf {A}}
\newcommand\cb{\mathsf {B}}
\newcommand\cc{\mathsf {C}}

\newcommand\cf{\mathsf {F}}
\newcommand\ff{\mathsf {f}}
\newcommand\gggg{\mathsf {g}}
\newcommand\yy{\mathsf {y}}
\newcommand\cu{\mathsf {U}}

\newcommand\ck{\mathsf {K}}
\newcommand\cl{\mathsf {L}}

\newcommand\cn{\mathsf {N}}
\newcommand\crr{\mathsf {R}}

\newcommand\ct{\mathsf {T}}

\title{Codensity: Isbell duality, pro-objects, compactness and accessibility}
\author{Ivan Di Liberti$^\dag$}

\thanks{$^\dag$ The author has been supported through the grant 19-00902S from the Grant Agency of the Czech Republic.}

\address{
Ivan \textsc{Di Liberti}: \newline
Department of Mathematics and Statistics\newline
Masaryk University, Faculty of Sciences\newline
Kotl\'{a}\v{r}sk\'{a} 2, 611 37 Brno, Czech Republic\newline
\href{mailto:diliberti@math.muni.cz}{\sf diliberti@math.muni.cz}
}

\begin{document}

\begin{abstract}
We study codensity monads $\ct$ induced by (mostly small, mostly dense) full subcategories $\ca \subset \ck$. These monads behave quite similarly, we show some connections with the Isbell duality, pro-finite objects and compact spaces. We prove that they are quite unlikely to be accessible. Finally, we introduce the notion of generically idempotent monad and comment its properties.
\end{abstract}
\maketitle
\setcounter{tocdepth}{1}
\tableofcontents

\section{Introduction}

The general theory of codensity monads was quite well developed in \cite{leinster2013codensity}, even if they appear in the literature in various instances and different flavors. Codensity monads can be defined in several ways, in this article we shall have a strong preference for the following:

\begin{defn}[Codensity monad]
Let $\ff : \ca \to \cb$ be a functor for which the right Kan extension $\ran_{\ff} \ff$ exists, then the universal property of the Kan extension induces naturally a monad structure $\ct_\ff = (\ran_{\ff} \ff, \eta, \mu)$, which is called the \textbf{codensity monad} of $\ff$. A description of the unit and the multiplication can be found in \cite{leinster2013codensity}[Sec. 5].
\begin{center}
\begin{tikzcd}
\ca \arrow[dd, "\ff"'] \arrow[rr, "\ff"]                         &  & \cb \\
                                                                 &  &     \\
\cb \arrow[rruu, "\ran_\ff \ff" description, dashed, bend right=20] &  &    
\end{tikzcd}
\end{center}
\end{defn} 

\cite{leinster2013codensity} provides equivalent definitions and beautiful insights. We address the reader to Leinster's paper for an introduction rich of examples and remarks. \cite{devlin2016codensity} provides a much more complete and terse introduction to this subject, including a concrete description of the unit and the counit \cite{devlin2016codensity}[Chap. 5]. Our treatment will be very far from being concrete and takes very seriously the presentation of \cite{leinster2013codensity}[Sec. 5], pushing \textit{the formal perspective} on codensity monads as far as possible.

\begin{rem}
There are several examples of codensity monads, indeed the Semantics-Structure adjunction \[\mathsf{Sem}: \mathsf{Mon}^\circ(\cb) \leftrightarrows \mathsf{Cat}_{/\cb}: \mathsf{Str} \]\cite{kanextensiondubuc}[pg. 74] proves that every monad $\ct: \ck \to \ck$ is the codensity monad of its forgetful functor $\mathsf{U}_\ct : \mathsf{Alg}(\ct) \to \ck$. For this reason, being the codensity monad of \textit{some} functor is not a deep property for a monad. Things change dramatically when one studies \textbf{codensity monads of full subcategories} $\ca \subset \ck$. Those will be our object of interest. 
\end{rem}

\begin{notat} \label{notation}
Given a functor $\ff: \ca \to \cb$ we will call $\ct_\ff$ the monad structure naturally induced on the endofunctor $\ran_\ff \ff$. Given a monad $\ct$ we will call $\cf_\ct \dashv \cu_\ct$ the induced free-forgetful adjunction over its category of algebras. In the specific case of a codensity monad we might refer to $\cf_{\ct_\ff} \dashv \cu_{\ct_\ff}$ as $\cf_{\ff} \dashv \cu_{\ff}$ to use a more clean notation. This means for example that \[\ran_\ff \ff = \ct_\ff =   \cu_\ff \circ \cf_\ff. \]
\end{notat}

\begin{defn} We say that a monad is \textbf{generically idempotent} if it is the codensity monad of a full subcategory up to natural isomorphism of monads.
\end{defn}

\begin{exa}[Idempotent monads are generically idempotent]
Let us justify this definition with an example. Recall that a monad is idempotent when its multiplication $\mu: \ct^2 \to \ct$ is an isomorphism. It is well known that an idempotent monad $\ct: \cb \to \cb$ individuates a reflective subcategory $\cu_\ct: \mathsf{Alg}(\ct) \hookrightarrow \cb$. Using the Structure-Semantics adjunction, we have that \[\ran_{\cu_\ct} \cu_\ct \cong \ct, \] this shows that $\ct$ is the codensity monad of its full subcategory of algebras, that is: idempotent monads are generically idempotent.
\end{exa}

\subsection{Motivating examples}

The reader should not think that generically idempotent monads are idempotent, or very close to be idempotent. Most of the relevant examples of generically idempotent monads are not idempotent at all. We list a couple examples, these display most of the properties that we would like to understand of codensity monads of full (relevant) subcategories.

\begin{center}

\begin{tabular}{lllll}
 Name & Monad   &  Subcategory & Algebras  & Reference  \\
 Ultrafilter & $\mathcal{U}: \Set \to \Set$  & $\mathsf{FinSet} \subset \Set$  &  CompT$_2$\footnotemark[1] & \cite{leinster2013codensity}    \\
 Double dual & $(\_)^{**}: \mathsf{Vect} \to \mathsf{Vect}$ & $\mathsf{FinVect} \subset \mathsf{Vect}$& LinComp\footnotemark[2] &  \cite{leinster2013codensity}     \\
 Giry & $\mathsf{G}: \mathsf{Meas} \to \mathsf{Meas}$ & $\mathsf{Convex} \subset  \mathsf{Meas}$  & Prob\footnotemark[3] & \cite{avery2016codensity}   
\end{tabular}
\end{center}

\footnotetext[1]{Compact Hausdorff spaces.}
\footnotetext[2]{Linearly Compact Spaces.}
\footnotetext[3]{Probability spaces.}
All of the previous examples are quite celebrated and understood\footnote[4]{The last one has recently proved to be less understood than expected. The interested reader should give a look to \cite{2017arXiv170700488S,2019arXiv190700372S,2019arXiv190703209S}.} in mathematics and we choose not to introduce them. \cite{leinster2013codensity} is a perfect introduction for the non-expert reader. We will deal with four main questions and hopefully show that the theory of Kan extensions is the perfect tool to handle this kind of questions.

\subsection{Questions}
\begin{enumerate}{}
\item What is the precise link between codensity monads and \textbf{double dualization}?
\item Why do some algebras for a codensity monad very often look like \textbf{pro-finite objects}?
\item Is there any relation between the algebras for a codensity monad and \textbf{compact Hausdorff spaces}?
\item Why are codensity monads often  \textbf{not accessible}?
\item Can we \textbf{characterize} codensity monads of full subcategories?
\end{enumerate}

We claim that it is possible to describe a general theory of codensity monads that have this kind of properties and those coincide with generically idempotent monads.

\subsubsection{Codensity and Isbell duality} In section \ref{isbellsec} we answer to the first two questions putting codensity monads in the context of Isbell duality. This perspective appear implicitly already in \cite{leinster2013codensity}[Sec 2.]. Isbell duality is a deep example of dual adjunction. There are several evidences that it is related with many dualities of the form $\mathsf{Alg}^\circ \cong \mathsf{Geom}$. We find that unveiling the connection with the Isbell duality has itself a conceptual content, moreover, even if the answer to the first question was essentially already in \cite{leinster2013codensity}, the answer to the second question was far from being there and is deeply based on the connection with Isbell duality. 

\subsubsection{Codensity and compact Hausdorff spaces}
It happens quite often that Pro-finite objects admit an additional structure of compact Hausdorff space. That's the case of pro-finite sets, pro-finite groups, pro-finite lattices. In Sec. \ref{compactness} we show that the algebras for the codensity monad of \textit{finite stuff} in a category of \textit{stuff} always have such a structure. This propagates to pro-finite objects using the result of  Sec.\ref{isbellsec}. We also discuss the possible mismatch between the notions of finite and finitely presentable object.

\subsubsection{Codensity and accessibility}
This section deals with preservation of colimits. We show that very often codensity monads cannot be cocontinuous, moreover when they are, they are willing to be the identity. They are more likely to be accessible, but we show that there are strong obstructions also to that.

\subsubsection{Characterization of generically idempotent monads}
In the last section we justify the definition of generically idempotent monad showing that a monad is generically idempotent if and only if it is idempotent on a limit-dense subcategory (in a proper sense) and conjecture a characterization of generically idempotent monads.

\begin{rem} The paper contains an Appendix collecting useful facts about Kan extensions. We make an extensive use of the results contained in the Appendix and they represent the core of our technical toolbox.
\end{rem}


\section{Isbell duality}\label{isbellsec}

This section sets a link between codensity monads of \textbf{full small subcategories} $\ca \subset \ck$ and Isbell duality.  In order to do so we have to recast Isbell duality in a suitable language that will deliver the main theorem (\ref{main}) as a trivial corollary.

\begin{thm*}[\ref{main}]
The codensity monad of the Yoneda embedding is isomorphic to the monad induced by the Isbell adjunction, \[ \ct_\yy \cong \mathsf{Spec} \circ \mathcal{O}. \]
\end{thm*}

 The Isbell duality is a syntax-semantics kind of duality that was introduced by Lawvere in \cite{lawvere1986taking}, where the author credits Isbell \cite{isbell} for the general idea. The core of the section is the following motto: \textit{the Isbell duality is the theory of the codensity monad of the Yoneda embedding}. We dedicate the following subsection to recast Isbell duality. 

 \begin{rem}
 After the first version of this paper was submitted on the arXiv, Georges Charalambous informed the author that he has independently shown the enriched version of the restricted Isbell adjunction (\ref{restrictedisbell}) in an unpublished work \cite{charalambous} of 2016 together with some other similar results contained in this section.
 \end{rem}

\subsection{A brief introduction to Isbell duality}
\begin{rem}[The (co)presheaf construction]
Let $\ca$ be a small category. We will denote by $\yy_\ca: \ca \to \Set^{\ca^\circ}$ the usual presheaf construction, aka the Yoneda embedding. Analogously, we call $\yy_\ca^\sharp: \ca \to (\Set^{\ca})^\circ$ the copresheaf construction. Observe that $\yy_\ca^\sharp$ coincides precisely with $(\yy_{\ca^\circ})^\circ$. Also, recall that these constructions are the free completion under (co)limits of $\ca$\footnote[5]{Be careful, the presheaf construction is the completion under colimits, while the co-presheaf construction is the completion under limits.}. Both these categories are complete and cocomplete. We feel free to use the abuse of notation $\yy_\ca = \yy$, when it does not generate any confusion.
\end{rem}

\begin{rem}[The (co)nerve construction]
Given a span where $\cb$ is a cocomplete category and $\ca$ is small, there is an andjunction
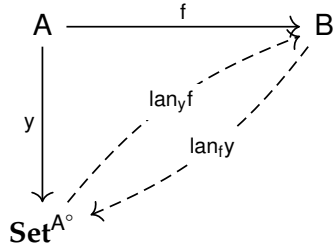
\begin{center}
\begin{tikzcd}
\ca \arrow[ddd, "\yy"'] \arrow[rrr, "\ff"]                                       &  &  & \cb \arrow[lllddd, "{\lan_\ff \yy}" description, dashed, bend left=15] \\
                                                                             &  &  &                                                             \\
                                                                             &  &  &                                                             \\
\Set^{\ca^\circ} \arrow[rrruuu, "\lan_{\yy} \ff" description, dashed, bend left=15] &  &  &                                                            
\end{tikzcd}
\end{center}
$\lan_{\yy} \ff \dashv \lan_{\ff} \yy$ where the right adjoint coincides with the functor $\cb(\ff \_, \_)$. $\lan_{\ff} \yy$ is called \textit{the nerve of} $f$, while $\lan_{\yy} \ff$ is called \textit{geometric realization}. This construction was introduced by Kan in \cite{10.2307/1993103}[Sec 3] in the special case of simiplicial sets, hence the name. A proof that the couple $(\lan_{\yy} \ff , \lan_{\ff} \yy)$ are adjoint functors can be found in many references and belongs to the realm of formal category theory, we will include a proof of it in the Appendix \ref{lan_ylan_f}.

 Analogously to the covariant case, when $\cb$ is complete, the conerve construction is the adjunction provided in the following diagram by the right Kan extensions.

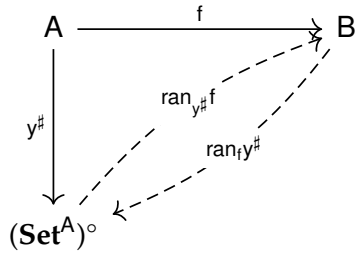
\begin{center}
\begin{tikzcd}
\ca \arrow[ddd, "\yy^\sharp"'] \arrow[rrr, "\ff"]                                       &  &  & \cb \arrow[lllddd, "{\ran_\ff \yy^\sharp}" description, dashed, bend left=15] \\
                                                                             &  &  &                                                             \\
                                                                             &  &  &                                                             \\
(\Set^{\ca})^\circ \arrow[rrruuu, "\ran_{\yy^\sharp} \ff" description, dashed, bend left=15] &  &  &                                                            
\end{tikzcd}
\end{center}
Observe that in this case the adjunction switches left with right: $\ran_{\yy^\sharp} f \vdash \ran_{\ff} \yy^\sharp$, and  $\ran_{\ff} \yy^\sharp$ corresponds to $\cb(\_, \ff \_)$. 
\end{rem}

\begin{rem}[Isbell duality] \label{isbell}
The Isbell duality \cite{lawvere1986taking}[7] is a special case of the nerve construction, when the relevant span is the following.
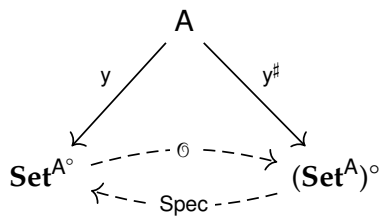
\begin{center}
\begin{tikzcd}
                                                                & \ca \arrow[ldd, "\yy"'] \arrow[rdd, "\yy^\sharp"] &                                                                   \\
                                                                &                                             &                                                                   \\
\Set^{\ca^\circ} \arrow[dashed, rr, "\mathcal{O}" description, bend left=15] &                                             & (\Set^\ca)^\circ \arrow[dashed, ll, "\mathsf{Spec}" description, bend left=15]
\end{tikzcd}
\end{center}
Applying the nerve construction, one gets the adjunction $\lan_{\yy} \yy^\sharp \dashv \lan_{\yy^\sharp} \yy$, while the (co)nerve construction gives $\ran_{\yy} \yy^\sharp \dashv \ran_{\yy^\sharp} \yy$. Very surprisingly, these two adjunctions are in fact the same. In order to see this, one can observe that $\ran_{\yy} \yy^\sharp$ and $\lan_{\yy} \yy^\sharp $ are indeed the same functor (up to natural isomorphism) because they are both cocontinuous and assume the same value on representables, \[\ran_{\yy} \yy^\sharp \circ  \yy \stackrel{\ref{simplify}}{\cong} \yy^\sharp \]
\[\lan_{\yy} \yy^\sharp \circ  \yy \stackrel{\ref{simplify}}{\cong} \yy^\sharp. \]
Thus we feel free to define $\mathcal{O}$ as $ \lan_\yy \yy^\sharp$ and $\mathsf{Spec} = \ran_{\yy^\sharp} \yy$.  As a result of the previous discussion we get the adjunction: \[\Set^{\ca^\circ}(\_, \yy \_) =  \lan_\yy \yy^\sharp = \ran_\yy \yy^\sharp = \mathcal{O} \dashv \mathsf{Spec} = \ran_{\yy^\sharp} \yy= \lan_{\yy^\sharp} \yy = \Set^{\ca}(\yy^\sharp,  \_). \]
\end{rem}

\begin{prop}[Isbell swaps adjunctions]\label{swap} Let $i: \ca \subset \ck$ be a small full dense subcategory of a cocomplete category. Then, the induced adjunction $\lan_\yy i = \cl : \Set^{\ca^\circ} \leftrightarrows \ck : \crr = \lan_i \yy$ turns into an adjunction \[\mathcal{O} \circ \crr  \dashv \cl \circ \mathsf{Spec}\] when composed with the Isbell duality.
\end{prop}
\begin{rem}
Observe that this is a non-trivial result because we are composing right adjoint with left ones. Also, observe that the left adjoint becomes a right one and vice versa.
\end{rem}
\begin{proof}
We use the characterization in \ref{left adjoint exist}. We will just show that $\lan_{\mathcal{O} \circ \crr} (\mathsf{id}_\ck)$ coincides with $\cl \circ \mathsf{Spec}$ and we omit the rest of the proof. The proof is a straight line of isomorphisms.
\begin{align*} 
\lan_{\mathcal{O} \circ \crr} (\mathsf{id}_\ck)  \cong &  \lan_{\mathcal{O} \circ \crr} (\lan_i i) \quad & \text{by }\ref{density} \\ 
 \cong & \lan_{\mathcal{O} \circ \crr \circ i} i   \quad & \text{by }\ref{compose}\\
\cong & \lan_{\mathcal{O} \circ \yy} i  \\
\cong &  \lan_{\yy^\sharp} i \quad & \text{by }\ref{simplify} \\
 \cong &  \lan_{\yy} i \circ \lan_{\yy^\sharp} \yy \quad & \text{by }\ref{pointwiselan}\\
\cong &  \cl \circ \mathsf{Spec}
\end{align*}

\end{proof}

\subsection{Isbell duality and double dualization}

 Finally, we come to the core of this section, we show the connection between Isbell duality and codensity monads. 
 
\begin{thm}\label{main}
The codensity monad of the Yoneda embedding is isomorphic to the monad induced by the Isbell adjunction. \[ \ct_\yy = \mathsf{Spec} \circ \mathcal{O}. \]
\end{thm}
\begin{proof}
We showed enough that this is relatively trivial. It is enough to apply \ref{pointwiseran} to \ref{isbell}. \[\mathsf{Spec} \circ \mathcal{O} \stackrel{\ref{isbell}}{=} \ran_{\yy^\sharp} \yy  \circ \ran_\yy \yy^\sharp \stackrel{\ref{pointwiseran}}{\cong} \ran_{\yy} \yy. \]
\end{proof}

\begin{rem}[The Isbell duality induces (almost) every codensity monad]\label{Leinster} Finally we can tell precisely where our work meets that of \cite{leinster2013codensity}[Sec 2.]. Let $i: \ca \subset \ck$ \textbf{be a small full dense subcategory of a cocomplete}\footnote[6]{This also implies that $\ck$ is complete. It can be shown in several ways, direct and indirect one, we choose the combination of \cite{WOOD1982538}[Thm. 7] and \cite{CTGDC_1986__27_2_109_0}[5.6].} \textbf{category}.  Under these assumptions the nerve construction $\lan_\yy i = \cl : \Set^{\ca^\circ} \leftrightarrows \ck : \crr = \lan_i \yy$  presents $\ck$ as a reflective subcategory of $\Set^{A^\circ}$. Indeed $\ca$ is dense if and only if the nerve is fully faithful \cite{isbell1960}[1.8]. We can put this information together with the Isbell duality relative to $\ca$, as shown by the diagram below.

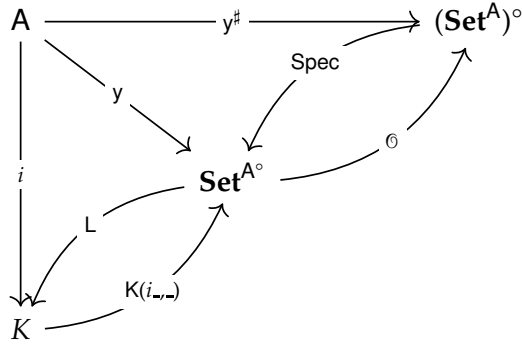
\begin{center}
\begin{tikzcd}
\ca \arrow[dddd, "i" description] \arrow[rrdd, "\yy" description] \arrow[rrrr, "\yy^\sharp" description]             &  &                                                                                                                  &  & (\Set^\ca)^\circ \arrow[lldd, "\mathsf{Spec}" description, bend right] \\
                                                     &  &                                                                                                                  &  &                                                            \\
                                                     &  & \Set^{\ca^\circ} \arrow[lldd, "\cl" description, bend right] \arrow[rruu, "\mathcal{O}" description, bend right] &  &                                                            \\
                                                     &  &                                                                                                                  &  &                                                            \\
K \arrow[rruu, "{\ck(i\_,\_)}" description, bend right] &  &                                                                                                                  &  &                                                           
\end{tikzcd}
\end{center}
As a consequence of Prop. \ref{swap}, we have that \[\mathcal{O} \circ  \ck(i\_,\_) \dashv \mathsf{L} \circ \mathsf{Spec}. \]  We claim that the monad $\mathsf{L} \circ \mathsf{Spec} \circ \mathcal{O} \circ  \ck(i\_,\_)$ induced over $\ck$ by this adjunction is precisely the codensity monad of the inclusion of $\ca$ in $\ck$. In order to show it, we use Prop. \ref{pointwiselan}, namely, \[\ran_i i \cong \ran_{\yy^\sharp} i \circ \ran_i \yy^{\sharp},\] and we show that the right adjoint $\ran_{\yy^\sharp} i$ coincides with the right adjoint $\mathsf{L} \circ \mathsf{Spec}$ on representables, hence the thesis. Indeed,

\[\ran_{\yy^\sharp} i \circ \yy^\sharp \cong i, \]
\[\mathsf{L} \circ \mathsf{Spec} \circ \yy^\sharp = \lan_\yy i \circ \ran_{\yy^\sharp} \yy \circ \yy^\sharp \stackrel{\ref{simplify}}{\cong} i. \]

This shows that the codensity monad of a full, small, dense subcategory in a cocomplete category is always induced by the Isbell duality via conjugation along the reflection \[\ct_i \cong \mathsf{L} \circ \mathsf{Spec} \circ \mathcal{O} \circ  \ck(i\_,\_).\]
This observation appears (very) implicitely  in \cite{leinster2013codensity}[Sec 2.], where the connection with Isbell duality is hidden by a more concrete presentation of the adjunction.
\end{rem}

\begin{rem}[On the ubiquity of Isbell]
The previous remark shows the ubiquity of the Isbell duality in the context of codensity monads of \textit{relevant} subcategories. Also, since this duality can be considered to be the archetypical example of double dualization, we feel that having this connection spelled out in details fills a gap in the exihisting literature and clarifies in which sense codensity monads are related to double dualizations. In this direction the recent paper \cite{2019arXiv190904950A} inspects the connection between this conceptual double dualization and a concrete double dualization induced by a monoidal closed structure on $\ck$.
\end{rem}

\subsection{Isbell duality and pro-objects}
In this subsection, let $\ck$ be a locally finitely presentable category and $\ck_\omega$ its full subcategory of finitely presentable objects\footnote[7]{Recall that this category is essentially small.} $i: \ck_\omega \subset \ck$. The aim of this section is to explain why some algebras of the codensity monad of $i$ looks like objects in the pro-completion of $\ck_\omega$.

\begin{exa}
The easiest example of this pattern, and indeed all the others look alike, is the category of sets. Indeed Set is locally finitely presentable and its finitely presentable objects are finite sets. It is well known that the algebras for the codensity monad $i: \Set_\omega \hookrightarrow \Set$ are compact Hausdorff spaces. In that case \textit{totally disconnected compact Hausdorff spaces are precisely pro-finite sets}. This shows that there is a very well characterized full subcategory of compact Hausdorff spaces that is equivalent to $\mathsf{Pro} \Set_\omega$.
\end{exa}

In the spirit of the previous example, we will provide an adjunction, \[\mathsf{Pro} \ck_\omega \leftrightarrows \mathsf{Alg}(\mathsf{T}_i) \] under the restriction that $\ck_\omega$ has  finite limits. Note that this hypothesis is verified in the case of finite sets. In order to do so, we have to give a restricted version of the Isbell duality.

\begin{rem}[Prolegomena to the restricted Isbell duality]
This remark is highly based on the theory of locally presentable and accessible categories, we mostly refer to \cite{adamekrosicky94}, the reader that is not familiar with the definitions is encouraged to read this reference and \cite{GUcentazzo}. Let $\ca$ be a small category with \textbf{finite limits and finite colimits}. Recall that in this case $\mathsf{Ind} \ca$ is a locally finitely presentable category \cite{GUcentazzo}[Th. 3.1]. The pro-completion $\mathsf{Pro} \ca$ of $\ca$ coincides with $(\mathsf{Ind} \ca^\circ)^\circ$ \cite{kashiwara2005categories}[6, pg 138] and thus is a full subcategory of the free completion of $A$ under limits. Being the opposite category of a locally presentable category, it is complete and cocomplete. In the diagram below we set the notation that we use for the rest of the section.

\begin{center}
\begin{tikzcd}
\mathsf{Ind} \ca \arrow[dd, "j" description]     &                                                                                                                                         & \mathsf{Pro} \ca \arrow[dd, "j^\sharp" description] \\
                                      & A \arrow[ld, "\yy" description] \arrow[rd, "\yy^\sharp" description] \arrow[lu, "i" description] \arrow[ru, "i^\sharp" description] &                                                 \\
\Set^{A^\circ} \arrow[rr, "\mathcal{O}" description, bend right=15] &                                                                                                                                         & (\Set^A)^\circ \arrow[ll, "\mathsf{Spec}" description, bend right=15]          
\end{tikzcd}
\end{center}
\end{rem}

\begin{rem}
Recall that $j = \lan_i \yy$, while $j^\sharp$ coincides with $\ran_{i^\sharp} \yy^\sharp$. Moreover, $\mathsf{Ind} \ca$ is reflective in the free completion $\Set^{\ca^\circ}$ , while the $\mathsf{Pro} \ca $ is coreflective in $(\Set^A)^\circ$. Also, recall that in the very special case in which $\ca$ has finite colimits, $\mathsf{Ind} \ca$ can be described as $\mathsf{Lex}(\ca^\circ, \Set)$ \cite{GUcentazzo}[2.10 (2)], similarly, $\mathsf{Pro} \ca = \mathsf{Lex}(\ca, \Set)^\circ$.
\end{rem}

\begin{prop}[Restricted Isbell duality]\label{restrictedisbell}
When $\ca$ has finite limits and finite colimits, the Isbell duality restricts to an adjunction between $\mathsf{Ind} \ca$ and $\mathsf{Pro} \ca$, \[ \mathscr{O}: \mathsf{Ind} \ca \leftrightarrows \mathsf{Pro} \ca : \mathscr{S}. \] Moreover $\mathscr{S}$ can be lifted to a functor $\cc: \mathsf{Pro} \ca  \to \mathsf{Alg}(\mathsf{T_i})$ in the sense clarified by the following diagram.

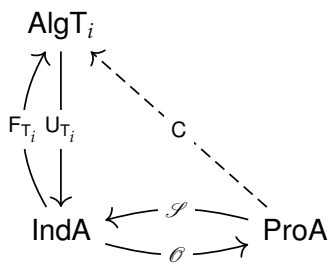
\begin{center}
\begin{tikzcd}
\mathsf{Alg}\ct_i \arrow[ddd, "\cu_{\ct_i}" description]                                                           &  &                                                                                                \\
                                                                                                           &  &                                                                                                \\
                                                                                                           &  &                                                                                                \\
\mathsf{Ind}\ca \arrow[uuu, "\cf_{\ct_i}" description, bend left] \arrow[rr, "\mathscr{O}" description, bend right=15] &  & \mathsf{Pro}\ca \arrow[ll, "\mathscr{S}" description, bend right=15] \arrow[lluuu, "\mathsf{C}" description, dashed]
\end{tikzcd}
\end{center}

\end{prop}
\begin{proof}
\begin{itemize}
	\item[]
	\item[Step 1] Let's start by constructing the adjunction. Since both $\mathsf{Ind} \ca$  and $\mathsf{Pro} \ca $ can be identified as the subcategory of some functors preserving a family of limits, it is enough to show that $\mathcal{O}(F)$ preserves any family of limits for any $F$. As surprising as it may sound, this is a straightforward verification, given the operative definition of $\mathcal{O}$.  Obviously the dual result can be shown for $\mathsf{Spec}$.

$$
  \begin{aligned}
    \mathcal{O}(X)(\lim c_j)
    & := \Set^{\ca^\circ}(X,\yy(\lim c_j))
    \\
    & \cong \Set^{\ca^\circ}(X, \lim \yy( c_j))
    \\
    & \cong  \lim  \Set^{\ca^\circ}(X, \yy( c_j))
    \\
    & =: \lim \mathcal{O}(X)( c_j).
  \end{aligned}
  $$
  \item[Step 2] Now we dedicate to the functor $\cc$. We know that the restricted Isbell duality $\mathscr{O} \dashv \mathscr{S}$ induces a monad over $\mathsf{Ind}(\ca)$, that is precisely the codensity monad of the inclusion $i: \ca \to \mathsf{Ind}(\ca)$. Since free-forgetful adjunction $\mathsf{Alg}\ct_i$ is terminal among those adjunctions that induce the restricted Isbell duality, there exists a diagonal functor lifting $\mathscr{S}$ along $\cu_{\ct_i}$.

	\end{itemize}
\end{proof}




Recall the notations of this subsection, let $\ck$ be a locally finitely presentable category and $\ck_\omega$ its full subcategory of finitely presentable objects $i: \ck_\omega \subset \ck$. We are now ready to provide an adjunction $\mathsf{Pro} \ck_\omega \leftrightarrows \mathsf{Alg}(\mathsf{T}_i),$ as in the example of compact Hausdorff spaces.

\begin{cor}\label{procodensity2}
Let $\ck$ be a locally finitely presentable category such that $\ck_\omega$ is closed under finite limits, then there is an adjunction between the algebras for the codensity monad over $i: \ck_\omega \hookrightarrow \ck$ and pro-(finitely presentable) objects,

 \[ \mathsf{Alg}(\mathsf{T}_i) \leftrightarrows  \mathsf{Pro} \ck_\omega: C\]
\end{cor}
\begin{proof}
This follows immediately from the previous proposition. Indeed when $\ck$ is locally finitely presentable, $\mathsf{Ind}(\ck_\omega)$ is equivalent to $\ck$, \cite{adamekrosicky94}[Thm 1.46].
\end{proof}

\subsection{$\mathsf{Ind}_\lambda$ and $\mathsf{Pro}_\lambda$ }

There is no reason to think that the restricted Isbell duality is specific to the case of finite limits, the adjunction extends to the $\mathsf{Ind}_\lambda$-completion precisely in the same way of the subsection above.

\begin{prop}[Restricted Isbell duality]\label{generalized restricted}
When $\ca$ is $\lambda$-complete and $\lambda$-cocomplete, the Isbell duality restricts to an adjunction between $\mathsf{Ind}_\lambda  \ca$ and $\mathsf{Pro}_\lambda  \ca$. \[ \mathscr{O}_\lambda: \mathsf{Ind}_\lambda \ca \leftrightarrows \mathsf{Pro}_\lambda \ca : \mathscr{S}_\lambda \]
\end{prop}

\begin{rem}[No-go equivalence] \label{nogo}
Observe that this duality cannot be an equivalence of categories, in fact  $\mathsf{Pro}_\lambda  \ca$ is the opposite of a locally presentable category, and thus cannot be locally presentable itself. The same argument works also for the presheaf-version of the Isbell duality. 
\end{rem}

\begin{cor}\label{noidentity} Let $\ca$ be a $\lambda$-complete and $\lambda$-cocomplete category. Then the following cannot happen simultaneously:
\begin{enumerate}
\item  $i :  \ca  \to \mathsf{Ind}_\lambda  \ca$ is codense.
\item  $i^\sharp :  \ca  \to \mathsf{Pro}_\lambda  \ca$ is dense.
\end{enumerate}
\end{cor}
\begin{proof}
It would contradict Rem. \ref{nogo}. In fact the monad and the codensity monad represent precisely the two possible compositions of the adjoint functors involved in the Isbell duality. If both of them are the identity (up to natural isomorphism), the adjunction is an equivalence.
\end{proof}

\begin{rem}[As sharp as it can be]\label{sharp}
Originally we were hoping to give a stronger statement, namely that in the hypotheses of Cor. \ref{noidentity} it is never that case that $i :  \ca  \to \mathsf{Ind}_\lambda  \ca$ is codense. This is not true, in fact given  an inaccessible cardinal $\lambda$  such that every lambda-complete ultrafilter is trivial, the full subcategory $\Set_\lambda$ of those sets of cardinality smaller than $\lambda$  is codense in $\Set$, this result is due to Isbell and a reference can be found in \cite{adamekrosicky94}[A.5]. This observation was pointed out by Jiří Rosický in a private communication.
\end{rem}

\section{Compactness}\label{compactness}

Let $\ck$ be a category of algebraic structures, say groups for the sake of simplicity. In this section we prove that the algebras for the codensity monads of finite structures $\ck_{\text{fin}}$ admit a structure of compact Hausdorff space. This intuition is very vague and should be contextualized. Given a category $\ck$, there is no natural notion of \textit{finite} object. Moreover, it would be wrong to choose the \textit{finitely presentable} as a candidate notion of \textit{finite}, this will be discussed at the end of the section. In the following remark we clarify what we mean by \textit{finite} object.

\begin{ass} In this section we work in the following assumptions. Let \[\cf: \Set \leftrightarrows  \ck : \cu \] be an adjunction where $\cu$ is the right adjoint. Denote by $\ck_{\text{fin}}$ the full subcategory of objects $k$ such that $\cu(k)$ is a finite set. We will use the following names for the corresponding functors,

\begin{center}
\begin{tikzcd}
\ck_{\text{fin}} \arrow[rr, "u" description] \arrow[dd, "j" description] &  & \text{FinSet} \arrow[dd, "i" description] \\
                                                                         &  &                     \\
\ck  \arrow[rr, "\cu" description]                                       &  & \Set                
\end{tikzcd}
\end{center}
\end{ass}

\begin{rem}
In the category of groups, finite objects in the sense of the previous remark correspond to finite groups when we choose $\cu$ to be the forgetful functor. Finitely presentable objects, instead, correspond to finitely presented groups in the sense of Universal Algebra. Yet, in many locally finitely presentable categories this notion of finite matches to the notion of finitely presentable, that's the case of sets, join semilattices, graphs...
\end{rem}

\begin{rem}
Let $\ck$ be a locally finitely presentable category with an object $k$ such that $\ck(k,\_): \ck \to \Set$ is faithful, strongly finitely accessible\footnote[8]{Recall, this means that the functor preserves directed colimits and finitely presentable objects.} and reflects finitely presentable objects, then we can assume that the two notions of \textit{finite} and \textit{finitely} presentable coincide. Such a $k$ would be a very special kind of finitely presentable object, indeed being finitely presentable only means that $\ck(k,\_)$ is finitely accessible.  This condition is met in all the examples mentioned at the end of the remark above, and clearly is not met in the case of groups. In the language of Universal Algebra this pattern can be found in the case of locally finite varieties.
\end{rem}

\begin{thm} \label{compactnessthm} The category of algebras of $\ct_j$ admits a functor $\Omega$ to compact Hausdorff spaces lifting the composition $\cu_j\cu$ along $\cu_i$. Moreover, if $\cu$ is faithful or conservative, so does the functor $\Omega$. Also, in this case $\Omega$ preserves limits\footnote[9]{In this theorem, we use the short form $\cu_{f}$ for the notation $\cu_{\ct_f}$ as explained in Not. \ref{notation}.}.

\end{thm}

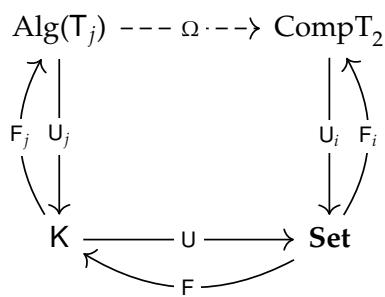
\begin{center}
\begin{tikzcd}{}
\text{Alg}(\ct_j) \arrow[ddd, "\cu_j" description] \arrow[rr, "\Omega" description, dashed, description]    &  & \text{CompT}_2 \arrow[ddd, "\cu_i" description]                                     \\
                                                                          &  &                                                                                        \\
                                                                          &  &                                                                                        \\
\ck \arrow[rr, "\cu" description] \arrow[uuu, "\cf_j" description, bend left] &  & \Set \arrow[ll, "\cf" description, bend left] \arrow[uuu, "\cf_i" description, bend right]
\end{tikzcd}
\end{center}

\begin{proof}
The proof goes in three steps. In the first two we construct $\Omega$, in the last one we show that $\Omega$ has all the desired properties.
\begin{itemize}
\item[Step 1] Call $\ct$ the monad $\cu\cu_j\cf_j\cf$. In the first step we assume to be provided with a morphism  of monads $\phi: \ct_i \Rightarrow \ct$, and we construct the dotted functor $\Omega$. This is quite easy to show in fact. If we have such a morphism of monads, we get a comparison functor $\cn: \mathsf{Alg}(\ct) \to \mathsf{Alg}(\ct_i).$ Recall that $\mathsf{Alg}(\ct_i)$ is precisely the category of compact Hausdorff spaces.
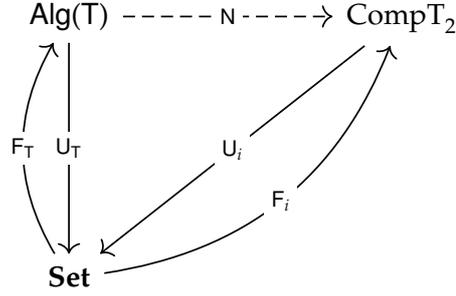
\begin{center}
\begin{tikzcd}
\mathsf{Alg}(\ct) \arrow[dddd, "\cu_\ct" description] \arrow[rrr, "\cn" description, dashed]             &  &  & \text{CompT}_2 \arrow[llldddd, "\cu_i" description] \\
                                                                                               &  &  &                                                        \\
                                                                                               &  &  &                                                        \\
                                                                                               &  &  &                                                        \\
\Set \arrow[rrruuuu, "\cf_i" description, bend right] \arrow[uuuu, "\cf_\ct" description, bend left] &  &  &                                                       
\end{tikzcd}
\end{center}

Now, we use the fact that the adjunction $\cf_T \dashv \cu_T$ is terminal among those adjunctions that induce the monad $\ct$, and thus we get another comparison functor $\cc:  \mathsf{Alg}(\ct_j) \to  \mathsf{Alg}(\ct) $ as indicated by the diagram below.

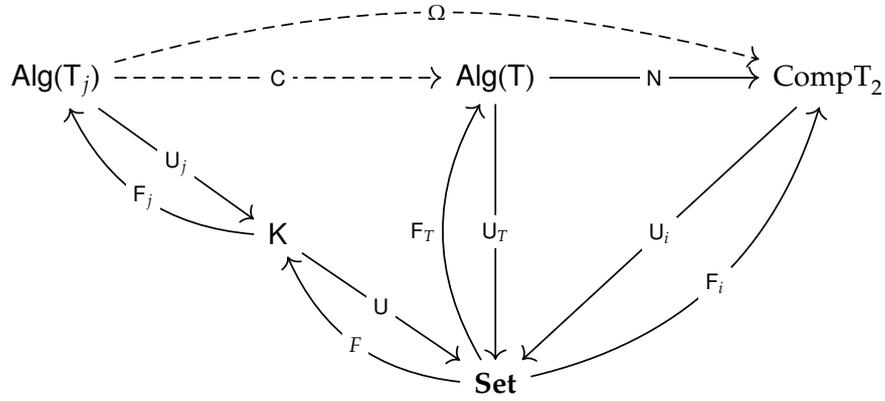
\begin{center}
\begin{tikzcd}
\mathsf{Alg}(\ct_j) \arrow[rrdd, "\cu_j" description] \arrow[rrrrrrr, "\Omega" description, dashed, bend left=15] \arrow[rrrr, "\cc" description, dashed] &  &                                                                              &  & \mathsf{Alg}(\ct) \arrow[dddd, "\cu_T" description] \arrow[rrr, "\cn" description]                                                  &  &  & \text{CompT}_2 \arrow[llldddd, "\cu_i" description] \\
                                                                                                                              &  &                                                                              &  &                                                                                                                             &  &  &                                                        \\
                                                                                                                              &  & \ck \arrow[rrdd, "\cu" description] \arrow[lluu, "\cf_j" description, bend left] &  &                                                                                                                             &  &  &                                                        \\
                                                                                                                              &  &                                                                              &  &                                                                                                                             &  &  &                                                        \\
                                                                                                                              &  &                                                                              &  & \Set \arrow[lluu, "F" description, bend left] \arrow[rrruuuu, "\cf_i" description, bend right] \arrow[uuuu, "\cf_T", bend left] &  &  &                                                       
\end{tikzcd}
\end{center}
The $\Omega$ we are looking for is the composite $\cn \circ \cc$.
\item[Step 2] The second step of the proof is to provide the  $\phi: \ct_i \Rightarrow \ct$ that we used in the previous step. Recall that $\ct$ coincides with $\cu \cu_j\cf_j \cf$ and thus corresponds to $\ran_{\cu\cu_j}(\cu\cu_j)$, now we claim that $\ran_{\cu\cu_j}(\cu\cu_j)$ coincides with $\ran_{\cu \circ j} ( \cu \circ j)$. In fact it is enough to follow the following chain of isomorphisms.

\begin{align*} 
\ran_{\cu\cu_j}(\cu\cu_j) \cong & \cu \circ  \ran_{\cu\cu_j}(\cu_j)  \\ 
\cong & \cu \circ \ran_{\cu}(\ran_{\cu_j}(\cu_j)) \\
\cong & \cu \circ \ran_{\cu}(\ran_j(j)) \\
\cong & \ran_{\cu \circ j} ( \cu \circ j)
\end{align*}

Thus, in order to finish the proof, we just need a morphism of monads $\phi: \ran_ii \Rightarrow \ran_{\cu \circ j} ( \cu \circ j)$. In the notation of the section, recall that $\cu \circ j$ coincides with $i \circ u$, thus we are looking for a map $\phi: \ran_ii \Rightarrow \ran_{i \circ u} ( i \circ u) $. Finally observe that since $\ran_{i \circ u} ( i \circ u)$ coincides with $ \ran_{i}(\ran_u(i \circ u))$ it is enough to provide a natural transformation $i \Rightarrow \ran_u(i \circ u) $, and then use the functoriality of the right Kan extension. Now simply recall that such a map $\phi: i \Rightarrow \ran_u(i \circ u) $ exist by the universal property of the right Kan extension.
\item[Step 3] Now, assume that $\cu$ is conservative or faithful, then $\Omega$ must be so, because it fits in the following diagram 
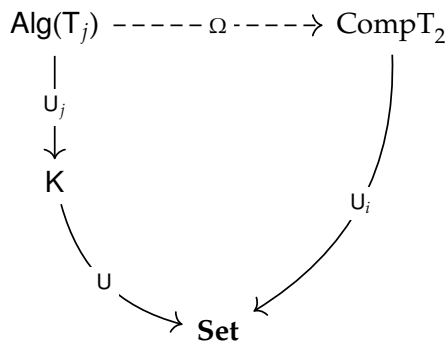
\begin{center}
\begin{tikzcd}
\mathsf{Alg}(\ct_j) \arrow[dd, "\cu_j" description] \arrow[rr, "\Omega" description, dashed] &      & \text{CompT}_2 \arrow[ldddd, "\cu_i" description, bend left] \\
                                                                                             &      &                                                              \\
\ck \arrow[rdd, "\cu" description, bend right]                                               &      &                                                              \\
                                                                                             &      &                                                              \\
                                                                                             & \Set &                                                             
\end{tikzcd}
\end{center}
where both the legs are faithful or conservative. A similar argument shows that it has to preserve limits using that $\cu_i$ create them, and create when $\cu$ does so.
\end{itemize}

\end{proof}

\begin{exa}

In the following examples finite structures are very relevant and algebras for their codensity monads have not been studied in deep detail.
\begin{enumerate}
\item Finite groups in $\mathsf{Grp}$.
\item $ \mathsf{FinBA}$ in $ \mathsf{CABA}$.
\item  $\mathsf{FinDL}$ in $\mathsf{Poset}$.
\item  $\mathsf{FinPoset}$ in $\mathsf{DL}$.
\item  $ \mathsf{FinJSL}_\bot$ (finite join semilattices with bottom) in $ \mathsf{JSL}_\bot$.
\item  $ \mathsf{FinAb}$ (finite abelian groups) $\mathsf{TorAb}$.
\end{enumerate}
Yet, it is well know that pro-finite groups have a natural structure of compact Hausdorff spaces and this now seems completely natural, in fact any pro-finite group is in a natural way an algebra for the codensity monad of finite groups (this would follow from an elaboration of \ref{restrictedisbell} and \ref{procodensity2}) and by the previous theorem, they inherit a compact Hausdorff structure. The same is true for finite abelian groups in Torsion Groups. The previous theorem shows that all the algebras for the codensity monads of finite strucutres admit such a compact Hausdorff structure and that this happens coherently with the natural compact Hausdorff structure that pro-object have.
\end{exa}

\begin{rem}[Johnstone's perspective]
The previous list of examples meets the content of  \cite{johnstone1986stone}[VI, 2.4]. It is quite tempting to believe in some connection between the results in this section and \cite{johnstone1986stone}[VI, 2.4-9], but the author did not manage to make them formal enough to disclose them. We choose to leave it as an open question to the interested reader.
\end{rem}

\begin{rem}[A comparison between the two notions of finite]
We warned the reader not to confuse \textit{finite} with \textit{finitely presentable}. It is our duty to provide at least one example in which the algebras for the codensity monad of finitely presentable objects are far from having a structure of compact Hausdorff space. That's the case of the double dualization monad in the category of vector spaces over the real line. As shown in \cite{leinster2013codensity}[7.8], in that context the algebras for the codensity monad are called linearly compact spaces. Despite the name, these spaces are not compact in a topological sense (except for very trivial examples), instead one could say that they are compact with respect to their algebraic nature.
\end{rem}

\section{Accessibility}

The main result of this section shows that if $\ck$ is a $\lambda$-presentable category whose $\lambda$-presentable objects are closed under $\lambda$-small limits, then the codensity monad of $\lambda$-presentable objects is very unlikely to be $\lambda$-accessible. Because of  Rem. \ref{sharp} it is more or less impossible to improve this result.

\begin{exa}
There are several examples of this behavior. The ultrafilter monad is very far from being accessible, the same is true for the Vietoris monad, or the Giry monad.
\end{exa}

\begin{defn}
Let $\ff: \ck \to \cb$ be a functor and let $i: \ca \subset \ck$ be a full  subcategory of $\ck$. We say that $\ff$ has \textbf{arity} $\ca$ iff we can recover $\ff$ from its restriction to $\ca$, i.e.
\[\ff \cong \lan_i (\ff \circ i ).\]
\end{defn}

\begin{rem}
If $\ck$ is a $\lambda$-accessible category, then being $\lambda$-accessible for a functor means precisely to have arity $\mathsf{Pres}_{\lambda}(\ck)$. If $\ck$ is a presheaf category, to be cocontinuous means precisely that the arity of $\ff$ is the Yoneda embedding.
\end{rem}

\begin{thm}\label{arity}
If the codensity monad of a full dense subcategory $\ca \subset \ck$ has arity $\ca$ , then it is the identity.
\end{thm}
\begin{proof}
\[T_i = \lan_i(T_i \circ i) =  \lan_i((\ran_i i) \circ i) \stackrel{\ref{simplify}}{\cong} \lan_i i \stackrel{\ref{density}}{\cong} \mathsf{id}.\]
\end{proof}

\begin{cor}
Let $\ca$ be $\lambda$-complete and $\lambda$-cocomplete. Then the following cannot happen simultaneously:
\begin{enumerate}
\item the codensity monad in the $\mathsf{Ind}_\lambda$-completion is $\lambda$-accessible.
\item the density comonad in the $\mathsf{Pro}_\lambda$-completion preserves $\lambda$-codirected limits.
\end{enumerate} 
\end{cor}
\begin{proof}
Apply  Cor. \ref{noidentity} to Thm. \ref{arity}.
\end{proof}

\begin{rem}
The last result shows that it is quite hard for the category of algebras of the codensity monad to be locally presentable, in fact this would force its monad to be accessible for some cardinal. We can't turn this into a theorem because the accessibility rank of the monad could be much higher then the accessibility rank of its category of algebras, on the other hand in practice this is a quite clear warning, one should never expect the category of algebras for the codensity monad to be essentially algebraic (unless we are in the reflective case itself, which trivializes the situation). 
\end{rem}

\section{Generically idempotent monads} 

In the first section we introduced the notion of generically idempotent monad, we showed that an idempotent monad is generically idempotent, but we also warned the reader not to believe that any generically idempotent monad is idempotent. 
So, \textit{how does a generically idempotent monad look like? Is this name well justified?} This final section is devoted to answer these two questions. We will see that in booth cases a central role is played by the full subcategory of fixed points. 

\begin{rem}
In this section we will build on a well known connection between fixed points of a monad and idempotency. Recall that an object $k \in \ck$ is fixed by the monad $(\ct, \eta, \mu$) if $\eta_k$ is an isomorphism. We will say that $\ct$ is idempotent \textit{at} $k$ if $\mu_k$ is an isomorphism, a monad is idempotent if it is idempotent at any object.
\end{rem}

\begin{prop} \label{idempotent/fixed} The following are equivalent. 
\begin{enumerate} 
\item  $\ct$ is idempotent at $k$,
\item $k$ is a fixed point for $\ct$.
\end{enumerate}
\end{prop}
\begin{proof}
Very similar to \cite{BOR2}[4.2.3].
\end{proof}

\begin{rem}
Let $\ct: \ck \to \ck$ be a monad. Its fixed points form a full subcategory $\mathsf{Fix}(\ct)$ of $\ck$, but also a full subcategory of its algebras, because every map between fixed objects is an algebra morphism. We consolidate this observation in the following diagram, where we set the notation for this section.

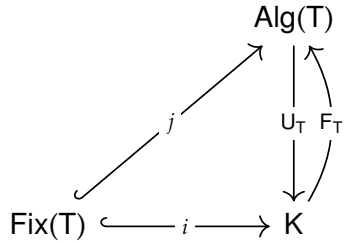
\begin{center}
\begin{tikzcd}
                                                                                         &  & \mathsf{Alg}(\ct) \arrow[ddd, "\cu_\ct" description] \\
                                                                                         &  &                                                      \\
                                                                                         &  &                                                      \\
\mathsf{Fix}(\ct) \arrow[rr, "i" description, hook] \arrow[rruuu, "j" description, hook] &  & \ck \arrow[uuu, "\cf_{\ct}" description, bend right]
\end{tikzcd}
\end{center}
\end{rem}

\begin{rem}[Generically idempotent monads have fixed points] In order to make the final claim of this section as natural as possible, we observe that the codensity monad of a full subcategory $l: \ca \subset \ck$ has many fixed points, in fact it fixes all the objects of $\ca$:

\[\ct_i \circ i = \ran_ii \circ i \stackrel{\ref{simplify}}{\cong} i. \]
{}
For this reason, when a monad has no fixed points, we can't expected it to be a generically idempotent. Also, observe that a generically idempotent monad is idempotent over its fixed points. A similar remark appears also in \cite{leinster2013codensity}[Sec 5.].
\end{rem}

\subsection{Generically idempotent monads are \textit{generically idempotent}}

 The following proposition justifies the notion of generically idempotent monads.

\begin{prop} \label{many fixed points} 

	If there exists a full subcategory $l:\ca \hookrightarrow \ck$ on which $\ct$ is idempotent, which is codense in the category of algebras, then $\ct$ is generically idempotent.

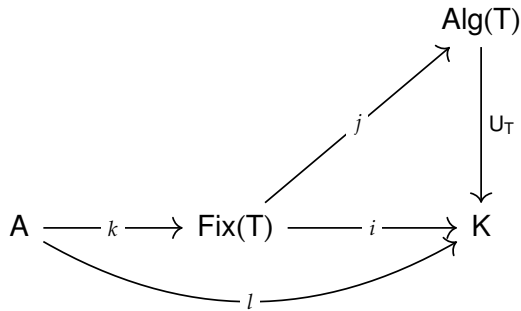
\begin{center}
\begin{tikzcd}
                                                                          &  &                                                                                &  & \mathsf{Alg}(\ct) \arrow[ddd, "\mathsf{U}_{\ct}"] \\
                                                                          &  &                                                                                &  &                                                       \\
                                                                          &  &                                                                                &  &                                                       \\
\ca \arrow[rrrr, "l" description, bend right] \arrow[rr, "k" description] &  & \mathsf{Fix}(\ct) \arrow[rr, "i" description] \arrow[rruuu, "j" description] &  & \ck                                                  
\end{tikzcd}
\end{center}

\end{prop}
\begin{proof}
It is enough to show that $\ct \cong \ran_l l$. In the notation of the previous diagram, we can assume that $\ran_{jk} jk = 1.$

\begin{align*} 
\ct \cong & \ran_{\cu_{\ct}} \cu_\ct   \\ 
\cong & \ran_{\cu_{\ct}} ( \cu_{\ct} \circ 1) \\
\cong & \ran_{\cu_{\ct}} ( \cu_{\ct} \circ \ran_{jk} (jk)) \\
\cong &\ran_{\cu_{\ct}} (  \ran_{jk} (\cu_{\ct} jk)) \\
\cong & \ran_l l.
\end{align*}

\end{proof}

\begin{rem}[The converse is somewhat true]\label{conversetrue}
Assume that $\ck$ is complete. When $\ct$ is generically idempotent $\ct \cong \ct_l$, there exists a full subcategory $l:\ca \hookrightarrow \ck$ on which $\ct$ is idempotent. This is given by the same full subcategory that induces the monad. For quite some time, the author has believed that this subcategory was codense in the category of algebras, but never managed to prove it and still does not have a counterexample. For the moment, the best that we can offer is to show that \textbf{ $\ca$ is limit-dense among $\ct$-algebras}. Let $k$ be a $\ct$-algebra with  structure map $\xi: \ct(k) \to k$. Since the algebra structure is a retraction for the unit, we can express $k$ as the following equalizer in $\ck$, 

\begin{center}
\begin{tikzcd}
\ct(k) \arrow[rr, "1" description, bend right] \arrow[rr, "\eta_k \circ \xi", bend left] &  & \ct(k)
\end{tikzcd}
\end{center}

\[k \cong \mathsf{Eq}(\eta_k \circ \xi, \mathsf{1}_{\ct(k)}).\]

\end{rem}
Now, recall that $\ct(k)$ is isomorphic to the canonical limit of the diagram $k/\ca$, because $\ct$ is generically idempotent. Using that limits commute with limits, we conclude that $k$ is a limit in $\ck$ of diagrams in which only objects in $\ca$ are involved.

\subsection{Towards a classification of generically idempotent monads}

We finish this paper with a conjecture about the shape of a generically idempotent monad. We conjecture that those are precisely those monads that can be recovered from their categories of fixed points. The subsection is devoted to make this conjecture formal.

Let $\ct \cong \ct_l$ be a generically idempotent monad. We can draw the following diagram

\begin{center}
\begin{tikzcd}
                                                                          &  &                                                                                &  & \mathsf{Alg}(\ct_l) \arrow[ddd, "\mathsf{U}_{l}"] \\
                                                                          &  &                                                                                &  &                                                       \\
                                                                          &  &                                                                                &  &                                                       \\
\ca \arrow[rrrr, "l" description, bend right] \arrow[rr, "k" description] &  & \mathsf{Fix}(\ct_l) \arrow[rr, "i" description] \arrow[rruuu, "j" description] &  & \ck                                                  
\end{tikzcd}
\end{center}

Recall that according to Not. \ref{notation} by $\mathsf{U}_{l}$ we actually $\mathsf{U}_{\ct_l}$. By the structure-semantics adjunction, we have $\ct \cong \ran_{\cu_\ct} \cu_\ct$, while by definition of generically idempotent we have $\ct \cong \ran_l l$. Our task now is to provide two arrows as in the diagram below.

\begin{center}
\begin{tikzcd}
         & \ct \arrow[ld, "\cong" description, no head] \arrow[rd, "\cong" description, no head] &                                                                  \\
\ran_l l &                                                                                       & \ran_{\cu_\ct} \cu_\ct \arrow[ld, "\psi" description, dashed, bend left] \\
         & \ran_i i \arrow[lu, "\phi" description, dashed, bend left]                            &                                                                 
\end{tikzcd}\end{center}

\begin{enumerate}
	\item[$\phi$)] Since $l=ik$, a map $\ran_i i \to \ran_l l$ is the same of a map $\ran_i i \to \ran_{ik} ik$. The latter is isomorphic to $\ran_i(\ran_k(ik)).$ Thus it is enough to find a map $i \to \ran_k(ik)$, and $\phi$ will be provided by the functoriality of the right Kan extension. Indeed such a map $i \to \ran_k(ik)$ exists and is given by the universal property of the right Kan extension.
	\item[$\psi$)] We proceed in a very similar way using $i=\cu_\ct j$.
\end{enumerate}

\begin{conj}
A monad is generically idempotent if and only if it is naturally isomorphic to the codensity monad of its fixed points. The isomorphism is given by the couple $\phi, \psi$ constructed in this subsection, which are one inverse to the other.
\end{conj}

\appendix

\section{Useful results about Kan extensions}

This appendix contain some facts about Kan Extension that are needed in the paper. We are including the proofs of some of them, for some others we use references. 

\begin{prop}\label{simplify}
Let $i: \ca \subset \ck$ be a full subcategory of $\ck$. Then, for all functors $\ff: \ca \to \cb$ for which the Kan extension $\lan_i \ff$ exists, one has \[(\lan_i \ff ) \circ i \cong \ff.\]
\end{prop}
\begin{proof}
\cite{borceux_1994}[3.7.3]
\end{proof}

\begin{prop}\label{compose}
$\lan_{\ff \circ \gggg} (\_) \cong \lan_{\ff}( \lan_{\gggg} (\_)).$
\end{prop}
\begin{proof}
$\lan_{\ff \circ \gggg}$ is the left adjoint to the precomposition $\_ \circ \ff \circ \gggg$. The latter right adjoint can be written as the composition of $\_ \circ \ff$ with $\_ \circ \gggg$. Now it is enough to observe that left adjoints compose.
\end{proof}

\begin{prop}\label{preserve}
Left adjoints preserve left Kan extensions, right adjoint preserve right Kan extensions.
\end{prop}
\begin{proof}
\cite{borceux_1994}[3.7.4]
\end{proof}

\begin{prop}\label{density}
The following are equivalent:
\begin{enumerate}
	\item $i: \ca \to \cb$ is dense;
	\item  $\lan_ii$ exists, is pointwise and coincides with $\mathsf{id}_\cb$ up to natural isomorphism.
\end{enumerate}
\end{prop}

\begin{prop} \label{left adjoint exist}
Let $R: \ca \to \cb$. The following are equivalent.

\begin{enumerate}
	\item $R$ is a right adjoint.
	\item The right Kan extension $\ran_R(1_\ca)$ exists and $R$ preserves it.
\end{enumerate}

Moreover, in that case $\ran_R(1_\ca)$ is the desired left adjoint. The dual statement involving left Kan extensions is true for left adjoints.
\end{prop}
\begin{proof}
\cite{borceux_1994}[3.7.6].
\end{proof}

\begin{prop}\label{lan_ylan_f}
Given a span where $\cb$ is a cocomplete category and $\ca$ is small, there is a andjunction
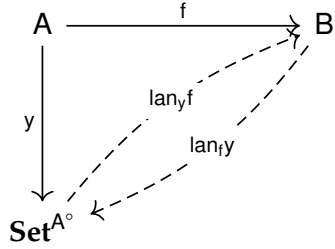
\begin{center}
\begin{tikzcd}
\ca \arrow[ddd, "\yy"'] \arrow[rrr, "\ff"]                                       &  &  & \cb \arrow[lllddd, "{\lan_\ff \yy}" description, dashed, bend left=15] \\
                                                                             &  &  &                                                             \\
                                                                             &  &  &                                                             \\
\Set^{\ca^\circ} \arrow[rrruuu, "\lan_{\yy} \ff" description, dashed, bend left=15] &  &  &                                                            
\end{tikzcd}
\end{center}
$\lan_{\yy} \ff \dashv \lan_{\ff} \yy$ where the right adjoint coincides with the functor $\cb(\ff \_, \_)$.
\end{prop}
\begin{proof}
It is quite easy to show that $\cb(\ff \_, \_)$ is the right adjoint of $\lan_{\yy} \ff $. In fact, this essentially follows by the Yoneda Lemma. We would like to show that there exists a natural isomorphism: \[\cb(\lan_yf(\_), \_ ) \cong \Set^{\ca^\circ}(\_, \cb(\ff \_, \_) ).\] This exists if and only if there exists one when we compose with the Yoneda embedding (this is essentially the content of the Yoneda lemma). Thus we can evaluate the previous line on the representables.

\[
\mathsf{B}(\lan_yf(y(\_)), \_ ) \stackrel{\ref{simplify}}{\cong}  \cb(f(\_), \_ ) \stackrel{\text{Yo}}{\cong} \Set^{\ca^\circ}(y(\_), \cb(\ff \_, \_) ).
\]
Now, we have to show that $\lan_\ff \yy$ is isomorphic to $\cb(\ff \_, \_)$. In order to do so, we use \ref{left adjoint exist} and prove that $\lan_{\lan_{\yy} \ff}(1) \cong \lan_\ff \yy$. This is a straightforward computation.

\[\lan_{\lan_{\yy} \ff}(1) \cong \lan_{\lan_{\yy} \ff}(\lan_\yy \yy) \cong \lan_{\ff} \yy.\]
\end{proof}

\begin{prop}   \label{pointwiselan}
In the diagram below, assume that $\ca$ is small and $\cc$ is cocomplete. Then the Kan extension $\lan_gf$ is isomorphic to the composition $\lan_{\yy_\ca}f \circ \lan_{g}\yy_\ca$.
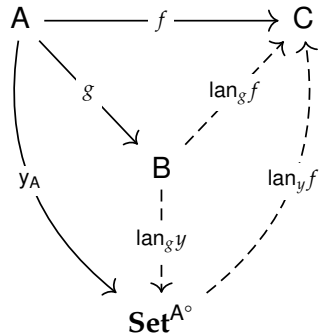
\begin{center}
\begin{tikzcd}
\ca \arrow[rr, "f" description] \arrow[rdd, "g" description]   \arrow[rdddd, "\yy_\ca" description, bend right]&                                                                                        & \cc \\
                                                           &                                                                                        &   \\
                                                           & \cb \arrow[dd, "\lan_gy" description, dashed] \arrow[ruu, "\lan_gf" description, dashed] &   \\
                                                           &                                                                                        &   \\
                                                           & \Set^{\ca^\circ} \arrow[ruuuu, "\lan_yf" description, dashed, bend right]                &  
\end{tikzcd}
\end{center}
\end{prop}
\begin{proof}
By the universal property of the presheaf construction, $\lan_{\yy_A}f$ is the left adjoint to the nerve of $f$. Being a left adjoint, it preserves all left Kan extensions, thus \[\lan_{\yy_\ca}f \circ \lan_{g}\yy_\ca \cong  \lan_{g}(\lan_{\yy_\ca}f \circ \yy_A) \cong  \lan_{g}f.  \]
\end{proof}

\begin{prop}  \label{pointwiseran}
Dually, one has that  $\ran_\gggg \ff$ is isomorphic to the composition $\ran_{\yy^\sharp_A}\ff \circ \ran_{\gggg}\yy^\sharp_A$.
\end{prop}

\section*{Acknowledgements}
 The author is grateful to Ji\v{r}\'i Rosick\'{y} for his observations about the content of Rem.\ref{sharp}, to Ji\v{r}\'i Adámek for having read a draft of this preprint providing very useful comments and to Fosco Loregian for having introduced the author to the beauty of formal methods in category theory. The author is indebted with Peter Arndt for showing enthusiasm about the content of this paper and many comments that led to an improvement of the presentation. Finally, we thank the anonymous referee for several suggestions that improved the exposition and the correction to Rem. \ref{conversetrue} where the author was originally claiming that the limit could be taken in the category of algebras.

\newpage

\bibliography{thebib}
\bibliographystyle{alpha}

\end{document}